\documentclass{article}
\usepackage{amsmath}
\usepackage{amssymb}

\newtheorem{proposition}{Proposition}
\newtheorem{lemma}{Lemma}

\newenvironment{proof}[1][Proof]{\noindent\textbf{#1.} }{\ \rule{0.5em}{0.5em}}
\input{tcilatex}
\input{tcilatex}

\begin{document}

\title{A ladder ellipse problem}
\author{Alan Horwitz \\
Professor Emeritus of Mathematics\\
Penn State Brandywine\\
25 Yearsley Mill Rd.\\
Media, PA 19063\\
alh4@psu.edu}
\date{10/7/15}
\maketitle

\begin{abstract}
We consider a problem similar to the well--known ladder box problem, but
where the box is replaced by an ellipse. A ladder of a given length, $s$,
with ends on the positive $x$ and $y$ axes, is known to touch an ellipse
that lies in the first quadrant and is tangent to the positive $x$ and $y$
axes. We then want to find the height of the top of the ladder above the
floor. We show that there is a value, $s=s_{0}$, such that there is only one
possible position of the ladder, while if $s>s_{0}$, then there are two
different possible positions of the ladder. Our solution involves solving an
equation which is equivalent to solving a $4$th degree polynomial equation.
\end{abstract}

\section{Introduction}

The well--known ladder box problem(see \cite{BE}, \cite{W}) involves a
ladder of a given length, say $s$ meters, with ends on the positive $x$ and $%
y$ axes, which touches a given rectangular box(a square in \cite{Z}) at its
upper right corner. One then wants to determine how high the top of the
ladder is above the floor. Other versions of problem(\cite{Z}) ask how much
of the ladder is between the wall(or floor) and the point of contact of the
ladder with the box. We ask similar questions in this note, but where the
box is replaced by an \textit{ellipse}, $E_{0}$, that lies in the first
quadrant and is tangent to the positive $x$ and $y$ axes at the points $(c,0)
$ and $(0,d)$. For example, consider the ellipse with equation $%
x^{2}+4y^{2}+2xy-8x-16y+16=0$, which is tangent to the positive $x$ and $y$
axes at the points $(4,0)$ and $(0,2)$. If the ladder has length $10$
meters, then how high is the top of the ladder above the floor and how many
positions of the ladder are possible ? One main difference here is that we
now allow the ladder to be tangent at \textit{any} point of $E_{0}$ rather
than just at the upper right corner of a rectangular box. We suppose that
the ladder touches the positive $x$ and $y$ axes at the points $(u,0)$ and $%
(0,v)$, respectively, and we call such a ladder admissible. We then want to
find $u$, which is the height of the top of the ladder above the floor. It
is not hard to show that the equation of $E_{0}$ must have the form

\begin{equation}
d^{2}x^{2}+c^{2}y^{2}+2Cxy-2cd^{2}x-2c^{2}dy+c^{2}d^{2}=0\text{,}  \label{1}
\end{equation}%
and that if the equation of $E_{0}$ is given by (\ref{1}), then $E_{0}$ is
tangent to the positive $x$ and $y$ axes at the points $(c,0)$ and $(0,d)$.
Note that for (\ref{1}) to be the equation of an ellipse, we need $%
c^{2}d^{2}-C^{2}>0$, which is equivalent to 
\begin{equation}
cd>\left\vert C\right\vert \text{.}  \label{5}
\end{equation}

We now assume throughout that $T$ is the triangle with vertices $(0,0),(u,0)$%
, and $(0,v)$ with $u,v>0$.

\textbf{Remark: }There is another way to look at this problem: Given an
ellipse, $E_{0}$, inscribed in a right triangle, $T$, suppose that we know
the length of the hypotenuse of $T$ and the points of tangency of $E_{0}$
with the other two sides of $T$. We want to find the lengths of the other
sides of $T$.

It is useful now to derive another form for the equation of $E_{0}$ which
depends on two parameters, which we denote by $w$ and $t$. The following
lemma and proposition were proven in \cite{H2} for the case when $T$ is the
unit triangle. Throughout we let $I$ denote the open interval $(0,1)$ and $%
I^{2}$ the unit square $=(0,1)\times (0,1)$.

\begin{lemma}
\label{medial}Let $V$ be the interior of the medial triangle of $T$ $=$
triangle with vertices at the midpoints of the sides of $T$. Then $V=\left\{
(x_{w,t},y_{w,t})\right\} _{(w,t)\in I^{2}}$, where

$x_{w,t}=\dfrac{1}{2}\dfrac{tu}{w+(1-w)t}$ and $y_{w,t}=\dfrac{1}{2}\dfrac{wv%
}{w+(1-w)t}$.
\end{lemma}

\begin{proof}
It follows easily that $(x,y)\in V$ if and only if 
\begin{eqnarray}
\dfrac{v}{2}-\dfrac{v}{u}x &<&y<\dfrac{v}{2}  \label{6} \\
0 &<&x<\dfrac{u}{2}\text{.}  \notag
\end{eqnarray}%
Now suppose that $(w,t)\in I^{2}$. We want to show that $\left(
x_{w,t},y_{w,t}\right) \in V$; First, $x_{w,t}-\dfrac{u}{2}=-\dfrac{1}{2}%
\dfrac{uw(1-t)}{w+(1-w)t}<0$ and $y_{w,t}-\dfrac{v}{2}=-\dfrac{1}{2}\dfrac{%
v\left( 1-w\right) t}{w+(1-w)t}<0$, which implies that $0<x_{w,t}<\dfrac{u}{2%
}$ and $y_{w,t}<\dfrac{v}{2}$; Also, $\dfrac{v}{u}x_{w,t}+y_{w,t}-\dfrac{v}{2%
}=\allowbreak \dfrac{1}{2}\dfrac{vwt}{w+(1-w)t}$ $>0$, and so $\dfrac{v}{2}-%
\dfrac{v}{u}x_{w,t}<y_{w,t}$, which implies that $\left(
x_{w,t},y_{w,t}\right) \in V$ by (\ref{6}). Conversely, suppose that $\left(
x_{w,t},y_{w,t}\right) \in V$. We want to show that $(w,t)\in I^{2}$;
Solving the system of equations 
\begin{eqnarray*}
\dfrac{1}{2}\dfrac{tu}{w+(1-w)t} &=&x \\
\dfrac{1}{2}\dfrac{wv}{w+(1-w)t} &=&y
\end{eqnarray*}%
yields the unique solution $w=w_{x,y}=\dfrac{2uy+2vx-uv}{2xv},t=t_{x,y}=%
\dfrac{2uy+2vx-uv}{2uy}$; Substituting $x=x_{w,t},y=y_{w,t}$ and using (\ref%
{6}) easily implies that

$0<w_{x,y},t_{x,y}<1$ and so $\left( w_{x,y},t_{x,y}\right) \in I^{2}$.
\end{proof}

\begin{proposition}
\label{P1}(i) Suppose that $E_{0}$ is an ellipse inscribed in $T$. Then the
equation of $E_{0}$ is given by 
\begin{gather}
(vw)^{2}x^{2}+(ut)^{2}y^{2}+2wt\left( 2w+2t-2wt-1\right) uvxy-2ut(vw)^{2}x
\label{3} \\
-2wv(ut)^{2}y+(uvwt)^{2}=0  \notag
\end{gather}

for some $(w,t)\in I^{2}$.

(ii) If $E_{0}$ is an ellipse with equation (\ref{3}) for some $0<t<1,0<w<1$%
, then $E_{0}$ is tangent to the three sides of $T$\ at the points $%
T_{1}=(ut,0),T_{2}=(0,vw)$, and $T_{3}=\left( \dfrac{ut(1-w)}{w+t-2wt},%
\dfrac{vw(1-t)}{w+t-2wt}\right) $.
\end{proposition}

\begin{proof}
Note that the denominator in both coordinates of $T_{3}$ is nonzero since $%
w+t-2wt>0$ holds for any $(w,t)\in I^{2}$ by the Arithmetic--Geometric Mean
inequality. First, suppose that $E_{0}$ is given by (\ref{3}) for some $%
(w,t)\in I^{2}$. Then $E_{0}$ has the form $Ax^{2}+By^{2}+2Cxy+Dx+Ey+F=0$,
where $AB-C^{2}$ easily simplifies to $\allowbreak 4(uvwt)^{2}\left(
1-w\right) \left( 1-t\right) {\large (}(1-t)w+t{\large )}>0$, and thus (\ref%
{3}) defines the equation of an ellipse. Also, $%
AE^{2}+BD^{2}+4FC^{2}-2CDE-4ABF=\allowbreak 16(uvwt)^{4}\left( 1-w\right)
^{2}\left( 1-t\right) ^{2}>0$, which implies that such an ellipse is
nontrivial. Now let $H(x,y)$ denote the left hand side of (\ref{3}). Since $%
H(T_{1})=\allowbreak H(T_{2})=H(T_{3})=0$, the three points $T_{1},T_{2}$,
and $T_{3}$ lie on $E_{0}$. Differentiating both sides of the equation in (%
\ref{3}) with respect to $x$ yields $\dfrac{dy}{dx}=D(x,y)$, where $D(x,y)=-%
\dfrac{vw}{ut}\dfrac{-\allowbreak 2uy\left( 1-w\right) t^{2}+\allowbreak
u\left( vw-2wy+y\right) t-v\allowbreak wx}{\left(
2w^{2}vx-2vwx-uy+vwu\right) t-vwx\left( 2w-1\right) }\allowbreak $; $%
D(T_{1})=\allowbreak 0=$ slope of horizontal side of $T$ and $%
D(T_{3})=\allowbreak -\dfrac{v}{u}=$ slope of the hypotenuse of $T$; When $%
x=0,y=w$, the denominator of $D(x,y)$ equals $0$, but the numerator of $%
D(x,y)$ equals $\allowbreak 2uvwt\left( 1-t\right) \left( 1-w\right) \neq 0$%
. Thus $E_{0}$ is tangent to the vertical side of $T$. For any simple closed
convex curve, such as an ellipse, tangent to each side of $T$ then implies
that that curve lies in $T$. Since it follows easily that $T_{1},T_{2}$, and 
$T_{3}$ lie on the three sides of $T$, that proves that $E_{0}$ is inscribed
in $T$. Second, suppose that $E_{0}$ is an ellipse inscribed in $T$. It is
well known \cite{C} that each point of $V$, the medial triangle of $T$, is
the center of one and only one ellipse inscribed in $T$, and thus the center
of $E_{0}$ lies in $V$. By Lemma \ref{medial}, the center of $E_{0}$ has the
form $(x_{w,t},y_{w,t}),(w,t)\in S$. Now it is not hard to show that each
ellipse given by (\ref{3}) also has center $(x_{w,t},y_{w,t})$. Since we
have just shown that (\ref{3}) represents a family of ellipses inscribed in $%
T$\ as $(w,t)$ varies over $I^{2}$, if $E_{0}$ were not given by (\ref{3})
for some $(w,t)\in I^{2}$, then there would be two ellipses inscribed in $T$
and with the same center. That cannot happen since each point of $V$ is the
center of only one ellipse inscribed in $T$. That proves (i). We have also
just shown that if $E_{0}$ is given by (\ref{3}), then $E_{0}$ is tangent to
the three sides of $T$\ at the points $T_{1},T_{2}$, and $T_{3}$, which
proves (ii).
\end{proof}

Now if we know that $E_{0}$ lies in the first quadrant and is tangent to the
positive $x$ and $y$ axes at the points $(c,0)$ and $(0,d)$, then $E_{0}$ is
inscribed in some triangle, $T$, with vertices $(0,0),(u,0)$, and $(0,v)$
with $u,v>0$. By Proposition \ref{P1}(ii), $c=ut$ and $d=vw$; Substituting
into (\ref{3}) yields

\begin{equation}
d^{2}x^{2}+c^{2}y^{2}+2cd\left( 2w+2t-2wt-1\right)
xy-2cd^{2}x-2c^{2}dy+c^{2}d^{2}=0\text{.}  \label{2}
\end{equation}%
Comparing (\ref{1}) and (\ref{2}) yields $C=cd(2w+2t-2wt-1)$, which implies
that 
\begin{equation}
w+t-wt=J\text{, }J=\dfrac{1}{2}\left( 1+\dfrac{C}{cd}\right) \text{.}
\label{4}
\end{equation}%
Note that by (\ref{5}) $cd>C$ and $cd>-C$, which implies that $0<J<1$; We
want to choose $(w,t)$ so that the ladder has the given length, $s$; Using $%
u=\dfrac{c}{t},v=\dfrac{d}{w}$, we have $s^{2}=u^{2}+v^{2}=\dfrac{c^{2}}{%
t^{2}}+\dfrac{d^{2}}{w^{2}}$, and since $w=\dfrac{J-t}{1-t}$\ from (\ref{4})
we have $s^{2}=f(t)$, where 
\begin{equation}
f(t)=\dfrac{c^{2}}{t^{2}}+\dfrac{d^{2}(1-t)^{2}}{(J-t)^{2}}\text{.}
\label{f}
\end{equation}%
If $t\in I$, then $\dfrac{J-t}{1-t}>0$ if and only if $t<J$; Also, if $t<1$,
then $1-\dfrac{J-t}{1-t}=\dfrac{1-J}{1-t}>0$; Thus we have 
\begin{equation}
w=\dfrac{J-t}{1-t}\in I\iff t<J\text{, where }t,J\in I\text{.}  \label{Jwt}
\end{equation}%
Thus for given $s$, using (\ref{f}), we want to solve the equation $%
f(t)=s^{2}$ for $t\in (0,J)$; For example, for the ellipse with equation $%
x^{2}+4y^{2}+2xy-8x-16y+16=0$, multiplying thru by $4$ yields the form of
the equation given in (\ref{1}), with $c=4$, $d=2$, and $C=4$; Suppose, say
that $s=10$. That gives $f(t)=\dfrac{16}{t^{2}}+\dfrac{4(1-t)^{2}}{\left(
3/4-t\right) ^{2}}$ and it is not hard to show that the equation $f(t)=100$
has two solutions $t_{1}=\dfrac{2}{3}$ and $t_{2}\approx 0.43$ in $(0,J)$, $%
J=\allowbreak \dfrac{3}{4}$; The corresponding $w$ values are then $w_{1}=%
\dfrac{J-t_{1}}{1-t_{1}}=\allowbreak \dfrac{1}{4}$ and $w_{2}=\dfrac{J-t_{2}%
}{1-t_{2}}\approx 0.56$, which gives $u_{1}=\dfrac{c}{t_{1}}=\allowbreak 6$, 
$v_{1}=\dfrac{d}{w_{1}}=8$, $u_{2}=\dfrac{c}{t_{2}}\approx \allowbreak
9.\allowbreak 35$, and $v_{2}=\dfrac{d}{w_{2}}\approx \allowbreak
3.\allowbreak 57$. The corresponding points where the ladder is tangent to $%
E_{0}$ are $T_{3,1}=\left( \dfrac{36}{7},\dfrac{8}{7}\right) $ and $%
T_{3,2}\approx (3.48,2.24)$; For this example there are \textit{two}
different positions of the ladder, which is analagous to what happens with
the ladder box problem. But are there always two different positions of the
ladder ? To help answer this, first we assume that there is an admissible
ladder of length $s$ which\ touches $E_{0}$, so it follows that the equation 
$f(t)=s^{2}$ has at least one solution in $(0,J)$; Since $%
\lim\limits_{t\rightarrow 0^{+}}f(t)=\lim\limits_{t\rightarrow
J^{-}}f(t)=\infty $, $f(t)=s^{2}$ must have at least two solutions in $(0,J)$%
, counting multiplicities. Now $f\,^{\prime }(t)=-2\left( \dfrac{c^{2}}{t^{3}%
}-\dfrac{d^{2}(1-t)(1-J)}{(J-t)^{3}}\right) $ and the function of $t,y=%
\dfrac{c^{2}}{t^{3}}$, is clearly decreasing on $(0,J)$; Since $\dfrac{d}{dt}%
\left( \dfrac{1-t}{(J-t)^{3}}\right) =\dfrac{(J-t)^{2}(3-J-2t)}{(J-t)^{6}}>0$
on $(0,J)$, the function of $t,y=\dfrac{d^{2}(1-t)(1-J)}{(J-t)^{3}}$ is
increasing on $(0,J)$; Thus the equation $\dfrac{c^{2}}{t^{3}}=\dfrac{%
d^{2}(1-t)(1-J)}{(J-t)^{3}}$ has at most one solution in $(0,J)$; Since $%
\lim\limits_{t\rightarrow 0^{+}}f\,^{\prime }(t)=\allowbreak -\infty $ and $%
\lim\limits_{t\rightarrow J^{-}}f\,^{\prime }(t)=\infty ,f\,^{\prime }$ has
at least one root in $(0,J)$; Hence $f\,^{\prime }$ has exactly one root,
say $t_{0}$, in $(0,J)$, and $\left\{ 
\begin{array}{ll}
f\,^{\prime }(t)<0 & \text{if }0<t<t_{0} \\ 
f\,^{\prime }(t)>0 & \text{if }t_{0}<t<J%
\end{array}%
\right. $; That in turn implies that $f$ is decreasing on $(0,t_{0})$ and is
increasing on $(t_{0},J)$ and so $f(t)=s^{2}$ has most two solutions in $%
(0,J)$; So we can conclude that $f(t)=s^{2}$ has exactly two solutions in $%
(0,J)$, \textit{counting multiplicities}. The only way that there would be
only one position of the ladder is if $f(t)-s^{2}$ has a double root in $%
(0,J)$; Can this actually happen ? To help answer this question, let $E_{R}=$
rightmost open arc of $E_{0}$ between the points, $P_{H}$ and $P_{V}$, on $%
E_{0}$ where the tangents are horizontal or vertical. Clearly there is an
admissible ladder tangent to $E_{0}$ at any point of $E_{R}$. As the point
of tangency approaches $P_{H}$ or $P_{V}$, $s$ approaches $\infty $. Hence
there is a unique value $s_{0}>0$ such that there is an admissible ladder of
length $s$ tangent to $E_{0}$ at any point of $E_{R}$ if and only if $s\geq
s_{0}$. How does one find $s_{0}$ ? $s_{0}=f(t_{0})$, where $t_{0}$ is the
unique root of $f\,^{\prime }$ in $(0,J)$ discussed above. For $s=s_{0}$,
there is only one position of the ladder, while if $s>s_{0}$, then there are
two different positions of the ladder. For the example above, $f\,^{\prime
}(t)$ has one root in $(0,J)$, $t_{0}\approx 0.58$; Then $%
s_{0}=f(t_{0})\approx \allowbreak 72$.

\textbf{Remark: }Solving $f(t)=s^{2}$\ is equivalent to solving the $4$th
degree polynomial equation%
\begin{equation}
p_{s}(t)=0,p_{s}(t)=(c^{2}-\allowbreak
s^{2}t^{2})(J-t)^{2}+d^{2}t^{2}(1-t)^{2}\text{.}  \label{p}
\end{equation}%
Note that one approach for solving the ladder box problem also involves
solving a $4$th degree polynomial equation.

\textbf{Remark: }Another way to solve this problem would be to use an affine
map to send $E_{0}$ to a circle, $C$, inscribed in a triangle, $T$, which is
now not necessarily a right triangle. Then the problem becomes: Suppose that
we know the length of a side, $c$, of a triangle, $T^{\prime }$, and we know
that a circle, $C$, is inscribed in $T^{\prime }$ and we know the points of
tangency of the other two sides, $a$ and $b$; Can one find the lengths of $a$
and $b$, and if yes, is the answer unique ?

\textbf{Special Case: }Not surprisingly, things simplify somewhat when the
ellipse, $E_{0}$, is a \textbf{circle}. In that case $d=c$, $C=0$, and $%
J=\allowbreak \dfrac{1}{2}$. The polynomial $p_{s}(t)$ from (\ref{p})
factors as a product of two quadratics:

$p_{s}(t)=-\dfrac{1}{4}{\large (}2\allowbreak \left( s-c\right) t^{2}-\left(
s-2c\right) t-c{\large )(}2\allowbreak \left( s+c\right) t^{2}-\left(
s+2c\right) t+c{\large )}$; It is then easy to show that the critical number 
$s_{0}$ of $f$ is given by $2\left( \sqrt{2}+1\right) c$, so that there are
two different positions of the ladder when $s>2\left( \sqrt{2}+1\right) c$.


\begin{thebibliography}{9}
\bibitem{BE} Patricia Baggett and Andrzej Ehrenfeucht, "The Ladder and Box
Problem: From Curves to Calculators, Oral presentation, History and Pedagogy
of Mathematics, 2012.

\bibitem{C} G. D. Chakerian, A Distorted View of Geometry, MAA, Mathematical
Plums, Washington, DC, 1979, 130--150.

\bibitem{H2} Alan Horwitz, \textquotedblleft Dynamics of ellipses inscribed
in triangles\textquotedblright , Journal of Science, Technology and
Environment, to appear.

\bibitem{W} David Wells, The \textquotedblleft Ladder
Problem\textquotedblright ; (1992), The Penguin Book of Curious and
Interesting Puzzles, p. 130.

\bibitem{Z} M. Zerger, \textquotedblleft The Ladder
Problem\textquotedblright , Mathematics Magazine 1987, 60 (4, Oct), pp.
239--242.
\end{thebibliography}
\end{document}